\documentclass[11pt]{article}

\usepackage{amsfonts,amssymb,amsmath,amsthm,epsfig,euscript}

\newtheorem{theorem}{Theorem}

\newtheorem{corollary}[theorem]{Corollary}
\newtheorem{definition}[theorem]{Definition}

\long\def\symbolfootnote[#1]#2{\begingroup
\def\thefootnote{\fnsymbol{footnote}}\footnote[#1]{#2}\endgroup}

\newcommand{\sg}{\sigma}
\newcommand{\ep}{\epsilon}

\newcommand{\cref}[1]{Corollary \ref{corollary:#1}}

\newcommand{\red}[1]{\mathrm{red}(#1)}

\newcommand{\tmch}[1]{\text{$\tau$-$\mathrm{mch}$}(#1)}
\newcommand{\umch}[1]{\text{$u$-$\mathrm{mch}$}(#1)}

\title{On a pattern avoidance condition for the
wreath product of cyclic groups with symmetric groups.}

\author{
Sergey Kitaev\footnote{The work presented here was supported by grant no. 090038011 from the Icelandic Research Fund.}\\
\small The Mathematics Institute\\[-0.8ex]
\small School of Computer Science\\[-0.8ex]
\small Reykjav\'{i}k University\\[-0.8ex]
\small IS-103 Reykjav\'{i}k, Iceland\\[-0.8ex]
\small \texttt{sergey@ru.is}
\and
Jeffrey Remmel\footnote{Partially supported by NSF grant DMS 0654060.} \\
\small Department of Mathematics\\[-0.8ex]
\small University of California, San Diego\\[-0.8ex]
\small La Jolla, CA 92093-0112. USA\\[-0.8ex]
\small \texttt{remmel@math.ucsd.edu}
\and
Manda Riehl\footnote{Supported by UWEC's Office of Research and Sponsored Programs.}\\
\small Department of Mathematics\\[-0.8ex]
\small University of Wisconsin, Eau Claire\\[-0.8ex]
\small Eau Claire, WI 54702-4004 USA\\[-0.8ex]
\small \texttt{riehlar@uwec.edu}
\and
}

\date{\small Submitted: Date 1;  Accepted: Date 2;
 Published: Date 3.\\
\small MR Subject Classifications: 05A15, 05E05}

\begin{document}
\maketitle

\begin{abstract}
\noindent

In this paper, we extend, to a non-consecutive case, the study of
the pattern matching condition on the wreath product $C_k \wr S_n$
of the cyclic group $C_k$ and the symmetric group $S_n$ initiated
in~\cite{KNRR}. The main focus of our paper is (colored) patterns of
length 2, although a number of enumerative results for longer
patterns are also presented. A new non-trivial bijective
interpretation for the Catalan numbers is found, which is the number
of elements in $C_2 \wr S_n$ bi-avoiding simultaneously (1-2,0~0)
and (1-2,0~1).

\end{abstract}

\section{Introduction}

The goal of this paper is to continue the study of pattern matching
conditions on the wreath product $C_k \wr S_n$ of the cyclic group
$C_k$ and the symmetric group $S_n$ initiated in~\cite{KNRR}. $C_k
\wr S_n$ is the group of $k^nn!$ signed permutations where there are
$k$ signs, $1=\omega^0$, $\omega$, $\omega^2$, $\ldots$,
$\omega^{k-1}$ where $\omega$ is a primitive $k$-th root of unity.
We can think of the elements $C_k \wr S_n$ as pairs $\gamma =
(\sg,\epsilon)$ where $\sg = \sg_1 \cdots \sg_n \in S_n$ and $\ep
=\ep_1 \cdots \ep_n \in \{1,\omega,\ldots, \omega^{k-1}\}^n$. For
ease of notation, if $\epsilon = (\omega^{w_1},\omega^{w_2}, \ldots,
\omega^{w_n})$ where $w_i \in \{0, \ldots, k-1\}$ for $i =1, \ldots,
n$, then we simply write $\gamma = (\sg,w)$ where $w = w_1 w_2
\cdots w_n$. Moreover, we think of the elements of $w = w_1 w_2
\cdots w_n$ as the {\em colors} of the corresponding elements of the
underlying permutation $\sg$.

Given a sequence $\sg = \sg_1 \cdots \sg_n$ of distinct integers,
let $\red{\sg}$ be the permutation found by replacing the
$i^{\textrm{th}}$ largest integer that appears in $\sg$ by $i$.  For
example, if $\sg = 2~7~5~4$, then $\red{\sg} = 1~4~3~2$.  Given a
permutation $\tau$ in the symmetric group $S_j$, define a
permutation $\sg = \sg_1 \cdots \sg_n \in S_n$ to have a {\em
$\tau$-match at place $i$} provided $\red{\sg_i \cdots \sg_{i+j-1}}
= \tau$.  Let $\tmch{\sg}$ be the number of $\tau$-matches in the
permutation $\sg$.  Similarly, we say that $\tau$ {\em occurs} in
$\sg$ if there exist $1 \leq i_1 < \cdots < i_j \leq n$ such that
$\red{\sg_{i_1} \cdots \sg_{i_j}} = \tau$.  We say that $\sg$ {\em
avoids} $\tau$ if there are no occurrences of $\tau$ in $\sg$.

We can define similar notions for words over a finite alphabet $[k]=
\{0,1, \ldots, k-1\}$. Given a word  $w = w_1 \cdots w_n \in [k]^*$,
let $\red{w}$ be the word found by replacing the $i^{\textrm{th}}$
largest integer that appears in $w$ by $i-1$.  For example, if $w =
2~7~2~4~7$, then $\red{w} = 0~2~0~1~2$.  Given a word $u \in [k]^j$
such that $\red{u}=u$, define a word $w \in [k]^n$ to have a {\em
$u$-match at place $i$} provided $\red{w_i \cdots w_{i+j-1}} = u$.
Let $\umch{w}$ be the number of $u$-matches in the word  $w$.
Similarly, we say that $u$ {\em occurs} in a word $w$ if there exist
$1 \leq i_1 < \cdots < i_j \leq n$ such that $\red{w_{i_1} \cdots
w_{i_j}} = u$. We say that $w$ {\em avoids} $u$ if there are no
occurrences of $u$ in $w$.

There are a number of papers on pattern matching and pattern
avoidance in $C_k\wr S_n$ (see~\cite{E,KNRR,Man,Man1,MW}).
For example, the
following pattern matching condition was studied in \cite{Man, Man1,MW}.
\begin{definition}\label{def1}
\begin{enumerate}
\item We say that an element $(\tau,u) \in C_k \wr S_j$ {\bf occurs} in
an element $(\sg,w) \in C_k \wr S_n$ if there are $1 \leq i_1 < i_2
< \cdots < i_j \leq n$ such that $\red{\sg_{i_1} \ldots \sg_{i_j}} =
\tau$  and $w_{i_p}=u_p$ for $p= 1, \ldots, j$.

\item We say that an element
$(\sg,w) \in C_k \wr S_n$ {\bf avoids} $(\tau,u) \in C_k \wr S_j$ if
there are no occurrences of $(\tau,u)$ in $(\sg,w)$.

\item If $(\sg,w) \in C_k \wr S_n$  and $(\tau,u) \in C_k \wr S_j$, then
we say that there is a {\bf $(\tau,u)$-match in $(\sg,w)$ starting
at position $i$} if $\red{\sg_{i} \sg_{i+1} \ldots \sg_{i+j-1}} =
\tau$ and $w_{i+p-1} = u_{p}$ for $p =1, \ldots, j$.
\end{enumerate}
\end{definition}

\noindent That is, an occurrence or match of $(\tau,u) \in C_k \wr
S_j$ in an element $(\sg,w) \in C_k \wr S_n$ is just an ordinary
occurrence or match of $\tau$ in $\sg$ where the corresponding signs
agree exactly. For example, Mansour \cite{Man1} proved via recursion
that for any $(\tau, u) \in C_k \wr S_2$, the number of
$(\tau,u)$-avoiding elements in $C_k \wr S_n$ is $\sum_{j=0}^n
j!(k-1)^j \binom{n}{j}^2$.  This generalized a result of Simion
\cite{Sim} who proved the same result for the hyperoctrahedral group
$C_2 \wr S_n$. Similarly, Mansour and West \cite{MW} determined the
number of permutations in $C_2 \wr S_n$ that avoid all possible 2
and 3 element set of patterns of elements of $C_2 \wr S_2$.  For
example, let $K_n^1$ be the number of $(\sg,\epsilon) \in C_2 \wr
S_n$ that avoid all the patterns in the set
$\{(1~2,0~0),(1~2,0~1),(2~1,1~0)\}$, $K_n^2$ be the number of
$(\sg,\epsilon) \in C_2 \wr S_n$ that avoid all the patterns in the
set $\{(1~2,0~1),(1~2,1~0),(2~1,0~1)\}$, and $K_n^3$ be the number
of $(\sg,\epsilon) \in C_2 \wr S_n$ that avoid all the patterns in
the set $\{(1~2,0~0),(1~2,0~1),(2~1,0~0)\}$. They proved that
\begin{eqnarray*}
K_n^1 &=& F_{2n+1}, \\
K_n^2 &=&   n! \sum_{j=0}^n \binom{n}{j}^{-1}, \ \mbox{and} \\
K_n^3 &=& n! +n!\sum_{j=1}^n \frac{1}{j}
\end{eqnarray*}
where $F_n$ is the $n$-th Fibonacci number.

In this paper, we shall consider the following pattern matching
conditions which were first considered in \cite{KNRR}.
\begin{definition}\label{def2}
Suppose that $(\tau,u) \in C_k \wr S_j$ and
$\red{u} =u$.
\begin{enumerate}
\item We say that $(\tau,u)$ {\bf bi-occurs} in
$(\sg,w) \in C_k \wr S_n$ if there are $1 \leq i_1 < i_2 < \cdots <
i_j \leq n$ such that $\red{\sg_{i_1} \cdots \sg_{i_j}} = \tau$  and
$\red{w_{i_1} \cdots w_{i_j}} = u$.

\item We say that an element
$(\sg,w) \in C_k \wr S_n$ {\bf bi-avoids} $(\tau,u)$ if there are no
bi-occurrences $(\tau,u)$ in $(\sg,w)$.

\item  We say that there is a {\bf $(\tau,u)$-bi-match in $(\sg,w) \in
C_k \wr S_n$ starting at
position $i$} if \\
$\red{\sg_{i} \sg_{i+1} \ldots \sg_{i+j-1}} = \tau$ and $\red{w_{i}
w_{i+1} \ldots w_{i+j-1}} = u$.
\end{enumerate}
\end{definition}

One can easily extend these notions to sets of elements of $C_k \wr
S_j$.  That is, suppose that $\Upsilon \subseteq C_k \wr S_j$ is
such that every $(\tau,u) \in \Upsilon$ has the property that
$\red{u} =u$. Then $(\sg,w)$ has an $\Upsilon$-bi-match at place $i$
provided $(\red{\sg_i \cdots \sg_{i+j-1}},\red{w_i \cdots
w_{i+j-1}}) \in \Upsilon$, $(\sg,w)$ has a bi-occurrence of
$\Upsilon$ if there exists $1 \leq i_1 < \cdots < i_j \leq n$ such
that $(\red{\sg_{i_1} \ldots \sg_{i_j}},\red{w_{i_1} \ldots
w_{i)j}}) \in \Upsilon$, and $(\sg,w)$ bi-avoids $\Upsilon$ if there
is no bi-occurrence of $\Upsilon$ in $(\sg,w)$.

We call the pair $(\tau,u)$ in Definition~\ref{def2} a {\em
pattern}. Moreover, to distinguish patterns in case of bi-occurrences
and bi-matches, we use for the former one dashes between the elements
of $\tau$. Thus, absence of a dash between any consecutive pair of
elements, say $x$ and $y$, in $\tau$ means that the elements in
$(\sg,w) \in C_k \wr S_n$ corresponding to $x$ and $y$ are adjacent.
For example, the pattern (1~2,0~0) bi-occurs in
$(\sg,w)=(1~3~2~4,1~2~2~2)$ at position $3$, whereas the pattern
(1-2,0~0) bi-occurs twice in $(\sg,w)$ additionally involving the
elements $\sg_3$ and $\sg_4$.

There are a number of natural maps which show that the problem of
finding the distribution of bi-matches or bi-occurrences for various
patterns is the same. That is, for any  $\sg = \sg_1 \cdots \sg_n
\in S_n$, we define the reverse of $\sg$, $\sg^r$, and the
complement of $\sg$, $\sg^c$, by
\begin{eqnarray*}
\sg^r &=& \sg_n \sg_{n-1} \cdots \sg_1 \ \mbox{and}\\
\sg^c &=& (n+1 -\sg_1) \cdots (n+1-\sg_n).
\end{eqnarray*}
Similarly, if $w =w_1 \cdots w_n \in [k]^n$, then we define the
reverse of $w$, $w^r$, and the complement of $w$, $w^c$, by
\begin{eqnarray*}
w^r &=& w_n w_{n-1} \cdots w_1 \ \mbox{and}\\
w^c &=& (k-1 -w_1) \cdots (k-1-w_n).
\end{eqnarray*}
We can then
consider maps $\phi_{a,b}:C_k \wr S_n \rightarrow C_k \wr S_n$ where
$\phi_{a,b}((\sg,w)) = (\sg^a,w^b)$ for $a,b \in \{r,c\}$. It
is easy to see that $(\sg,\tau)$ has a $(\tau,w)$  bi-match or bi-occurrence
if and only $(\sg^a,w^b)$ has a bi-match or bi-occurrence of
$(\tau^a,u^b)$.

Study of ($\Upsilon$-)bi-matches for patterns of length 2, i.e.
where $(\tau,u) \in C_k \wr S_2$ was the main focus of~\cite{KNRR}.
Such bi-matches are closely related to the analogue of rises and
descents in $C_k \wr S_n$ where we compare pairs using the product
order. That is, instead of thinking of an element of $C_k \wr S_n$
as a pair $(\sg_1 \cdots \sg_n,w_1 \cdots w_n)$, we can think of it
as a sequence of pairs $(\sg_1,w_1) (\sg_2,w_2) \cdots (\sg_n,w_n)$.
We then define a partial order on such pairs by the usual product
order. That is, $(i_1,j_1) \leq (i_2,j_2)$ if and only if $i_1 \leq
i_2$ and $j_1 \leq j_2$.

The outline of this paper is as follows. In section 2, we shall
consider bi-avoidance for patterns of length 2. Up to equivalence of
the maps induced by complementation and reversal, there are only 2
classes of patterns, namely those which are equivalent to (1-2,0~0)
and those that are equivalent to (1-2,0~1).  For patterns of the
form $(\tau,0^j)$, we derive formulas for the number of elements of
$C_k \wr S_n$ which bi-avoid $(\tau,0^j)$ from the formulas for the
number of elements of $S_n$ which avoid $\tau$, or the distribution
of the number of bi-occurrences of $(\tau,0^j)$ in elements of $C_k
\wr S_n$ from the corresponding formulas for the distribution of
occurrences of $\tau$ in $S_n$. For the pattern (1-2,0~1), we can
only find a formula for the number of elements of $C_2 \wr S_n$
which bi-avoid $(\tau,u)$. In section 3, we shall derive formulas
for the number of elements of $C_k \wr S_n$ which bi-avoid various
sets of patterns of length 2. In particular, we prove bijectively
(see Theorem~\ref{cat}) that the number of permutations in $C_2 \wr
S_n$ simultaneously bi-avoiding (1-2,0~0) and (1-2,0~1) is given by
the $(n+1)$-th Catalan number $C_{n+1}=\frac{1}{n+2}{2n+2 \choose
n+1}$.

\section{Bi-avoiding patterns}

In this section, we shall consider bi-avoidance with respect to patterns
of length 2. Up to the equivalence induced by the reversal and
complement maps for permutations and words, we need only consider
2 patterns, (1-2,0~0) and (1-2,0~1).

We start by considering the pattern $(1\mbox{-}2,0~0)$.

\begin{theorem}\label{mult} The number of permutations in $C_k \wr S_n$ bi-avoiding (1-2,0~0) is
given by $$\sum_{\begin{tiny}\begin{array}{c} i_1+\cdots +i_k=n\\
i_1\geq 0,\ldots,i_k\geq 0\end{array}\end{tiny}}{n\choose
i_1,\ldots,i_k}^2.$$\end{theorem}

\begin{proof} We first observe that elements of different colors are independent in permutations bi-avoiding (1-2,0~0), meaning that no
two elements with different colors can form a prohibited
configuration. Thus, assuming we have $i_{j+1}$ elements of color
$j$, $j=0,\ldots,k-1$, we can choose how to place these colors to
form a word $w$ in ${n\choose i_1,\ldots,i_k}$ ways. Then we can
choose the sets of elements $C_0, \ldots, C_{k-1}$ from
$\{1,\ldots,n\}$ which will correspond to the colors $0,1, \ldots,
k-1$ in $\sg$  in ${n\choose i_1,\ldots,i_k}$ ways. Finally, in
order to construct $(\sg,w)$ which bi-avoids the prohibited pattern,
we must place the elements of $C_i$ in the positions which are
colored $i$ in  decreasing order.
\end{proof}

The proof of Theorem~\ref{mult} suggests an obvious generalization
for patterns of arbitrary length with all possible dashes whose
elements are all colored by 0:

\begin{theorem}\label{mult1} Let $p$ be a pattern containing dashes in all places.
The number of permutations in $C_k \wr S_n$ bi-avoiding
($p$,0~$\cdots$0) is
given by $$\sum_{\begin{tiny}\begin{array}{c} i_1+\cdots +i_k=n\\
i_1\geq 0,\ldots,i_k\geq 0\end{array}\end{tiny}}A_{i_1}A_{i_2}\cdots
A_{i_k}{n\choose i_1,\ldots,i_k}^2$$ where $A_i$ is the number of
permutations in $S_i$ avoiding $p$. \end{theorem}

\begin{proof} A proof here is essentially the same as the proof of Theorem~\ref{mult}, except we can place $i$ elements of a permutation in $C_k \wr S_n$ of the
same color in any of $A_i$ ways. \end{proof}

It is well known that the number of $n$-permutations avoiding any
pattern of length 3 with dashes everywhere is given by the {\em
$n$-th Catalan number} $C_n=\frac{1}{n+1}{2n \choose n}$. As a
corollary to Theorem~\ref{mult1}, we have that for any pattern $p$
of length 3 with two dashes, the number of permutations in $C_k \wr
S_n$ bi-avoiding ($p$,0~0~0) is given by $$\sum_{\begin{tiny}\begin{array}{c} i_1+\cdots +i_k=n\\
i_1\geq 0,\ldots,i_k\geq 0\end{array}\end{tiny}}\frac{{2i_1 \choose
i_1}{2i_2 \choose i_2}\ldots {2i_k \choose
i_k}}{(i_1+1)(i_2+1)\cdots (i_k+1)}{n\choose i_1,\ldots,i_k}^2.$$

One can generalize even further Theorems~\ref{mult} and~\ref{mult1},
by talking on distribution of patterns. Indeed, assuming we know the
number $A_{i,j}$ of $i$-permutations containing $j$ occurrences of a
pattern $p$ (with dashes everywhere), we can write down the number
of permutations in $C_k \wr S_n$ with $j$ bi-occurrence of
($p$,0~$\cdots$0):
$$\sum_{\begin{tiny}\begin{array}{c} i_1+\cdots +i_k=n\\j_1+\cdots
+j_k=j\\ i_1\geq 0,\ldots,i_k\geq 0\\ j_1\geq 0,\ldots,j_k\geq
0\end{array}\end{tiny}}A_{i_1,j_1}A_{i_2,j_2}\cdots
A_{i_k,j_k}{n\choose i_1,\ldots,i_k}^2.$$ For example, the
distribution of the pattern 1-2 is the same as the distribution of
{\em inversions} in permutations (which coincides with the
distribution of 2-1), so one can extract the numbers $A_{i,j}$ as
the coefficients to $q^j$ in $\prod_{m=1}^{i}(1+q+\cdots +q^{m-1})$
and substitute them in the last formula to get the distribution of
the number of bi-occurrences of (1-2,0~0) in $C_k \wr S_n$.

Next we consider the pattern (1-2,0~1) in the case where $k=2$.

The number of permutations in $C_2 \wr S_n$ bi-avoiding (1-2,0~1) is
shown in~\cite[page 19]{Sim} to be equal to $\sum_{j=0}^{n}j!{n
\choose j}^2$. However, in Theorem~\ref{pat2} below we provide an
independent derivation of the exponential generating function in
this case.

\begin{theorem}\label{pat2} The exponential generating function for the number of permutations in $C_2 \wr S_n$ bi-avoiding the pattern
(1-2,0~1) is given by
$$A(x)=\frac{e^{\frac{x}{1-x}}}{1-x}.$$\end{theorem}

\begin{proof} Let $A_n$ denote the number of $n$-permutations in $C_2 \wr S_n$ bi-avoiding the pattern (1-2,0~1).
If an $(n+1)$-permutation contains the element 1 colored by the
color 1, then there are no restrictions for placing this element,
thus giving $(n+1)A_n$ possibilities. On the other hand, if an
$(n+1)$-permutation $(\sg,w)$ contains the element 1 colored by the
color 0 in position $i+1$, then if $(\sg,w)$ is to bi-avoid
(1-2,0~1), then every element to the right of 1 must be colored with
color 0 and these elements can be arranged in any of $(n-i)!$ ways.
Moreover, it immediately follows that no instance of a bi-occurrence
of (1-2,0~1) exists where the first element is to the left of 1, and
the second is to the right of 1. Thus it follows that $(\sg,w)$
bi-avoids (1-2,0~1) if and only if there is no bi-occurrence of
(1-2,0~1) in $(\sg_1 \cdots \sg_i,w_1 \cdots w_i)$. Thus, in the
case where 1 colored by 0 and is in position $i+1$, we have
$(n-i)!{n \choose i}A_i=n!A_i/i!$ possibilities where the binomial
coefficient is responsible for choosing the elements to the left of
1, and placing the remaining elements  to the right of 1. To
summarize, we obtain
$$A_{n+1}=n!\sum_{i=0}^{n}\frac{A_i}{i!}+(n+1)A_n.$$
Multiplying both parts of the equation above by $x^n/n!$ and summing
over all $n\geq 1$, we have
$$-A_1+A'(x)=-A_0+A(x)/(1-x)+xA'(x)+A(x)-A_0$$
leading to the differential equation
$$A'(x)=\frac{2-x}{(1-x)^2}A(x)$$
with the initial condition $A_0=1$ as the empty permutation bi-avoids
(1-2,0~1). The solution to this differential equation is
$$A(x)=\frac{e^{\frac{x}{1-x}}}{1-x}.$$ \end{proof}

\section{Bi-avoidance for sets of patterns}

In this section, we shall prove a variety of results for the number
of elements of $C_k \wr S_n$ that bi-avoid certain sets of patterns
of length 2. For any set $\Upsilon \subseteq C_k \wr S_j$ such that
$red(u) =u$ for all $(\tau,u) \in \Upsilon$, we let
$Av_{n,k}^\Upsilon$ denote the number of elements of $C_k \wr S_n$
which bi-avoid $\Upsilon$.

We start with a few simple results on sets of patterns $\Upsilon$ where
the bi-avoidance of $\Upsilon$
forces certain natural conditions on the possible
sets of signs for elements of $C_k \wr S_n$.

\begin{theorem}\label{thm:signs}\ \\

\begin{enumerate}
\item If $\Upsilon_1 = \{(1\mbox{-}2,0~0),(2\mbox{-}1,0~0)\}$, then $Av_{n,k}^{\Upsilon_1} = \binom{k}{n} n! n!$ for all $n \geq 1$ and $k \geq 1$.

\item If $\Upsilon_2 = \{(1\mbox{-}2,1~0),(2\mbox{-}1,1~0)\}$, then $Av_{n,k}^{\Upsilon_2} = \binom{n+k-1}{n} n!$ for all $n \geq 1$ and $k \geq 1$.

\item If $\Upsilon_3 = \{(1\mbox{-}2,0~0),(1\mbox{-}2,1~0),(2\mbox{-}1,0~0), (2\mbox{-}1,1~0)\}$, then $Av_{n,k}^{\Upsilon_3} = \binom{k}{n} n!$ for all $n \geq 1$ and $k \geq 1$.

\item If $\Upsilon_4 = \{(1\mbox{-}2,0~1),(1\mbox{-}2,1~0),(2\mbox{-}1,0~1), (2\mbox{-}1,1~0)\}$, then $Av_{n,k}^{\Upsilon_3} = k n!$ for all $n \geq 1$ and $k \geq 1$.

\end{enumerate}
\end{theorem}
\begin{proof}

For (1), it is easy to see that $(\sg,w) \in C_k \wr S_n$ bi-avoids $\Upsilon_1$ if
and only if all the signs are pairwise distinct. Thus there are
$\binom{k}{n}$ ways to pick the $n$ signs and then you have $n!$ ways to arrange those $n$ signs and $n!$ ways to pick $\sg$.

For (2), it is easy to see that $(\sg,w) \in C_k \wr S_n$ bi-avoids
$\Upsilon_2$ if and only $0 \leq w_1 \leq \cdots \leq w_n \leq k-1$.
Thus there are $\binom{n+k-1}{n}$ ways to pick the $n$ signs in this
case and there are $n!$ ways to pick $\sg$.

For (3), it is easy to see that $(\sg,w) \in C_k \wr S_n$ bi-avoids
$\Upsilon_2$ if and only $0 \leq w_1 < \cdots <  w_n \leq k-1$. Thus
there are $\binom{k}{n}$ ways to pick the $n$ signs in this case and
there are $n!$ ways to pick $\sg$.

For (4), it is easy to see that $(\sg,w) \in C_k \wr S_n$ bi-avoids
$\Upsilon_2$ if and only $0 \leq w_1 = \cdots =  w_n \leq k-1$. Thus
there are $k$ ways to pick the $n$ signs in this case and there are
$n!$ ways to pick $\sg$.
\end{proof}

If $\Upsilon \subseteq S_j$ is a set of permutations of $S_j$, we
let
\begin{equation}\label{Avgf}
\phi^{\Upsilon}(x,t) = \sum_{n \geq 0}  \frac{t^n}{n!} Av_n^{\Upsilon}
\end{equation}
where $Av_n^{\Upsilon}$ is the number of permutations $\sg \in S_n$ such
that $\sg$ avoids $\Upsilon$. Similarly, if
$\Upsilon \subseteq C_k \wr S_j$ is such that for all
$(\sg,u) \in \Upsilon$, $red(u) =u$, then we
let
\begin{equation}\label{Avkgf}
\theta_k^{\Upsilon}(x,t) = \sum_{n \geq 0}  \frac{t^n}{n!} Av_{n,k}^{\Upsilon}.
\end{equation}

We let $\Upsilon_0 = \{(\sg,0^j): \sg \in \Upsilon\}$. Then we have
the following.

\begin{theorem}\label{thm:product}
 Let $\Upsilon \subseteq S_j$ be any set of permutations of $S_j$ and
$\Upsilon_0 = \{(\sg,0^j): \sg \in \Upsilon\}$. Then if
$\Gamma = \Upsilon_0 \cup \{ (1\mbox{-}2,1~0), (2\mbox{-}1,1~0)\}$,
 \begin{equation}\label{eq:product}
\theta_k^{\Gamma}(t) = (\phi^{\Upsilon}(t))^k
\end{equation}
for all $k \geq 1$.
\end{theorem}
\begin{proof}
It is easy to see that if $(\sg,u) \in C_k \wr S_n$ and $(\sg,u)$
bi-avoids $\{ (1\mbox{-}2,1~0), (2\mbox{-}1,1~0)\}$, then it must be
the case that $0\leq u_1 \leq u_2 \leq \cdots \leq u_n \leq k-1$.
Thus suppose that $w = 0^{a_1} 1^{a_2} \cdots (k-1)^{a_k}$ where
$a_1 + \cdots + a_k = n$ and $a_i\geq 0$ for $1\leq i\leq k$. Then
clearly the number of $\sg \in S_n$ such that $(\sg,w)$ bi-avoids
$\Upsilon_0$ is
$$\binom{n}{a_1,\ldots,a_k} Av^{\Upsilon}_{a_1} \cdots Av^{\Upsilon}_{a_k}.$$
That is, the binomial coefficient allows us to choose the elements
$C_i$ of $\{1, \ldots, n\}$ that correspond to the constant segment
$i^{a_{i+1}}$ in $w$. Then we only have to arrange the elements of
$C_i$ so that it avoids $\Upsilon$ which can be done in
 $Av^{\Upsilon}_{a_{i+1}}$ ways. Thus it follows
that
\begin{equation*}
Av^{\Gamma}_{n,k} = \sum_{a_1 + \cdots +a_k =n; a_i \geq 0}
\binom{n}{a_1,\ldots,a_k} Av^{\Upsilon}_{a_1} \cdots Av^{\Upsilon}_{a_k}.
\end{equation*}
or equivalently
\begin{equation*}
\frac{Av^{\Gamma}_{n,k}}{n!} =
\sum_{a_1 + \cdots +a_k =n; a_i \geq 0} \prod_{i=1}^k \frac{Av^{\Upsilon}_{a_i}}{a_i!}
\end{equation*}
which implies (\ref{eq:product}).
\end{proof}

We immediately have the following corollary.

\begin{corollary} For any $k \geq 1$, the number of elements of $C_k \wr S_n$ which bi-avoids \\
$\Gamma_1 = \{(1\mbox{-}2,0~0), (1\mbox{-}2,1~0), (2\mbox{-}1,1~0)\}$ or
$\Gamma_2 = \{(2\mbox{-}1,0~0), (1\mbox{-}2,1~0), (2\mbox{-}1,1~0)\}$ is
$k^n$.
\end{corollary}
\begin{proof}
Let $\Upsilon_1 =\{1~2\}$ and  $\Upsilon_2 =\{2~1\}$.
Then clearly, $Av^{\Upsilon_i}_n =1$ for $i =1,2$ so that
$\phi_n^{\Upsilon_i}(t) = e^t$ for $i=1,2$. But then by Theorem \ref{thm:product},
$\theta_{n,k}^{\Gamma_i}(t) = (e^t)^k = e^{kt}$ so that
$Av_{n,k}^{\Gamma_i} = k^n$ for $i=1,2$.
\end{proof}

We can derive theorems analogous to Theorem \ref{thm:product}  for
the other sign conditions in Theorem \ref{thm:signs}. That is,
suppose that $\Upsilon$ is any set of patterns contained in $S_j$.
Then we let
$$\Upsilon_{i}= \{(\tau,0~1 \ldots (j-1)): \tau \in \Upsilon\}$$
and
$$\Upsilon_{d}= \{(\tau,u): \tau \in \Upsilon \ \& \ u \in D_j\}$$
where $D_j$ is the set of all permutations of $\{0,1, \ldots, j-1\}^*$. Then we have the following.

\begin{theorem} For any $\Upsilon \subset S_j$, let
\begin{eqnarray*}
\Gamma_1 &=& \Upsilon_0 \cup \{(1\mbox{-}2,0~1), (1\mbox{-}2,1~0),(2\mbox{-}1,0~1), (2\mbox{-}1,1~0)\}, \\
\Gamma_2 &=& \Upsilon_i \cup \{(1\mbox{-}2,1~0), (1\mbox{-}2,0~0),(2\mbox{-}1,1~0), (2\mbox{-}1,0~0)\}, \ \mbox{and} \\
\Gamma_3 &=& \Upsilon_{d} \cup \{(1\mbox{-}2,0~0),(2\mbox{-}1,0~0)\}. \\
\end{eqnarray*}
Then for all $k \geq 1$,
\begin{enumerate}
\item $Av_{n,k}^{\Gamma_1} = k Av_n^{\Upsilon}$ for all $n \geq 1$, \\

\item $Av_{n,k}^{\Gamma_2} = \binom{k}{n} Av_n^{\Upsilon}$ for all $n \geq 1$, and

\item $Av_{n,k}^{\Gamma_2} = \binom{k}{n} n! Av_n^{\Upsilon}$ for all
$n \geq 1$.
\end{enumerate}
\end{theorem}
\begin{proof}
For (1), note that for $(\sg,w) \in C_k \wr S_n$ to bi-avoid
$\{(1\mbox{-}2,0~1), (1\mbox{-}2,1~0),(2\mbox{-}1,0~1), (2\mbox{-}1,1~0)\}$,
$w = i^n$ for some $i \in \{0, \ldots, k-1\}$. Then
$(\sg,i^n) \in C_k \wr S_n$ bi-avoids  $\Upsilon_0$ if and only if
$\sg$ avoids $\Upsilon$. Thus $Av_{n,k}^{\Gamma_1} = k Av_n^{\Upsilon}$.

For (2), note that for $(\sg,w) \in C_k \wr S_n$ to bi-avoid
$\{(1\mbox{-}2,1~0), (1\mbox{-}2,0~0),(2\mbox{-}1,1~0),
(2\mbox{-}1,0~0)\}$, $w = w_1 \cdots w_n$ where $w_1 < \cdots <
w_n$. Then for any of the $\binom{k}{n}$ strictly increasing words
$w \in \{0,\ldots, k-1\}$, $(\sg,w) \in C_k \wr S_n$ bi-avoids
$\Upsilon_i$ if and only if $\sg$ avoids $\Upsilon$. Thus
$Av_{n,k}^{\Gamma_2} = \binom{k}{n} Av_n^{\Upsilon}$.

For (3), note that for $(\sg,w) \in C_k \wr S_n$ to bi-avoid
$\{(1\mbox{-}2,0~0),(2\mbox{-}1,0~0)\}$, $w = w_1 \cdots w_n$ where
the letters of $w$ are pairwise distinct. Then for any of the
$\binom{k}{n}n!$ words $w \in \{0,\ldots, k-1\}$ which have pairwise
distinct letters, $(\sg,w) \in C_k \wr S_n$ bi-avoids $\Upsilon_d$
if and only if $\sg$ avoids $\Upsilon$. Thus $Av_{n,k}^{\Gamma_3} =
\binom{k}{n} n! Av_n^{\Upsilon}$.
\end{proof}

Next we will prove two more results about $Av_{n,k}^\Upsilon$ for other sets of patterns $\Upsilon$ that contain $\{(1\mbox{-}2,1~0), (2\mbox{-}1,1~0)\}$

\begin{theorem}
Let
\begin{eqnarray*}
\Upsilon_1 &=& \{(1\mbox{-}2,0~1), (1\mbox{-}2,1~0), (2\mbox{-}1,1~0)\} \ \mbox{and} \\
\Upsilon_2 &=& \{(1\mbox{-}2,0~1), (1\mbox{-}2,1~0), (2\mbox{-}1,1~0),(2\mbox{-}1,0~0)\}.
\end{eqnarray*}
 Then
\begin{enumerate}
\item $Av_{n,k}^{\Upsilon_1} = \sum_{a_1+ \cdots + a_k =n, a_i \geq 0}
a_1! \cdots a_k!$ for all $n \geq 1$ and $k \geq 1$ and
\item $Av_{n,k}^{\Upsilon_1} = \binom{n+k-1}{k-1}$
for all $n \geq 1$ and $k \geq 1$.
\end{enumerate}
\end{theorem}
\begin{proof}
For (1), note that as in the proof of Theorem \ref{thm:product}, if
$(\sg,u) \in C_k \wr S_n$ and
$(\sg,u)$ bi-avoids $\{ (1\mbox{-}2,1~0), (2\mbox{-}1,1~0)\}$, then
it must be the case that $0\leq u_1 \leq u_2 \leq \cdots \leq u_n \leq k-1$.

Now suppose that $w = 0^{a_1} 1^{a_2} \cdots (k-1)^{a_k}$ where $a_1
+ \cdots + a_k = n$. Assume that $(\sg,w)$ also bi-avoids (1-2,0~1)
and $C_i$ is the set of elements of $\sg$ that correspond to the
signs $(i-1)^{a_i}$ in $w$. Then it follows that all the elements of
$C_1$ must be bigger than all the elements of $C_2$, all the
elements of $C_2$ must be bigger than all the elements of $C_3$,
etc.. Thus $C_1$ consists of the $a_1$ largest elements of
$\{1,\ldots,n\}$, $C_2$ consists of next $a_2$ largest elements of
$\{1,\ldots,n\}$, etc.. Then we can arrange the elements of $C_i$ in
any order in the positions corresponding to $(i-1)^{a_i}$ in $w$ to
produce a $(\sg,w)$ which bi-avoids $\Upsilon$.  Hence there are
$a_1! \cdots a_k!$ elements of the form $(\sg,w)$ which  bi-avoid
$\Upsilon$ in $C_k \wr S_n$. Thus (1) immediately follows.

For (2), observe that if in addition such $(\sg,w)$ also avoids
(2-1,0~0), then we must place the elements
of $C_i$ in increasing order. Hence $Av_{n,k}^{\Upsilon_2}$ is the
number of solutions of $a_1 + \cdots + a_k =n$ with $a_i \geq 0$ which
is well known to be $\binom{n+k-1}{k-1}$.
\end{proof}

\begin{theorem} Let $\Upsilon = \{(1\mbox{-}2,0~1),(1\mbox{-}2,1~0),
(2\mbox{-}1,0~0)\}$. Then for $n \geq 1$,
\begin{eqnarray*}
Av^{\Upsilon}_{n,1} &=& 1, \\
Av^{\Upsilon}_{n,2} &=& 2n, \ \mbox{and}\\
Av^{\Upsilon}_{n,k} &=& k + \sum{j=2}^{k-1} (k)\downarrow_j
\binom{n}{j} \ \mbox{for} \ k \geq 3,
\end{eqnarray*}
where $(k)\downarrow_0 =1$ and $(k)\downarrow_n = k(k-1) \cdots (k-n+1)$
for $n \geq 1$.
\end{theorem}
\begin{proof}
We shall classify the elements $(\sg,w) \in C_k \wr S_n$ which
bi-avoid $\Upsilon$ by the number of elements $s$ which follow $n$
in $\sg$. Now if $s=0$ so that $\sg$ ends in $n$, the fact that
$(\sg,w)$ bi-avoids both (1-2,0~1) and (1-2,1~0) means that all the
signs must be the same. That is, $w$ must be of the form $i^n$ for
some $i \in [k]$. But since $(\sg,w)$ must also bi-avoid (2-1,0~0),
$\sg = 1~2 \cdots (n-1)~n$ must be the identity. Thus there are $k$
choices for such $(\sg,w)$.

Now suppose there are $s$ elements following $n$ in $\sg$. Then again
the fact that $(\sg,w)$ bi-avoids both (1-2,0~1) and (1-2,1~0) means that all
the signs in $w$ corresponding to
 $\sg_1 \cdots \sg_{n-s} =n$ must be the same, say that sign is $i$. The fact
that $(\sg,w)$ avoids (2-1,0~0) means that (i) $\sg_1 \cdots
\sg_{n-s}$ must be in increasing order and (ii) all the signs
corresponding to $\sg_{n-s+1} \ldots \sg_n$ must be different from
$i$.  But then it follows from the fact that $(\sg,w)$ avoids both
(1-2,0~1) and (1-2,1~0) that all the elements in $\sg_1 \ldots
\sg_{n-s}$ must be greater than all the elements in $\sg_{n-s+1}
\cdots \sg_n$. Hence there are $k Av_{s,k-1}^{\Upsilon}$ such
elements if $k \geq 2$ and there are no such elements if $k =1$.
Thus it follows that $Av_{n,1}^{\Upsilon} = 1$ since $(1~2~\cdots
n,0^n)$ is the only element of $C_k \wr S_n$ which bi-avoids
$\Upsilon$.  For $k \geq 2$, we have
\begin{equation}\label{Urec}
Av_{n,k}^{\Upsilon} = k + \sum_{s=1}^{n-1} k Av_{s,k-1}^{\Upsilon}.
\end{equation}

In the case $n=2$, (\ref{Urec}) becomes
\begin{equation}\label{Urec2}
Av_{n,2}^{\Upsilon} = 2 + \sum_{s=1}^{n-1} 2  = 2n.
\end{equation}
and in the case $k=3$,
(\ref{Urec}) becomes
\begin{equation}\label{Urec3}
Av_{n,2}^{\Upsilon} = 3 + \sum_{s=1}^{n-1}3( 2s)  = 3+(3\cdot 2)\binom{n}{2}.
\end{equation}
In general, assuming that
\begin{equation*}
Av^{\Upsilon}_{n,k} = k + \sum_{j=2}^{k-1} (k)\downarrow_j \binom{n}{j},
\end{equation*}
it follows that
\begin{eqnarray*}
Av^{\Upsilon}_{n,k+1} &=& k +1 + \sum_{s=1}^{n-1} (k+1) Av^{\Upsilon}_{s,k} \\
&=& k+1 +  \sum_{s=1}^{n-1} (k+1) \left(k +\sum_{j=2}^{k-1} (k)\downarrow_j \binom{s}{j}\right) \\
&=& k+1 + \left(\sum_{s=1}^{n-1} (k+1)k\right) + \sum_{j=2}^{k-1} (k+1) (k) \sum_{s=1}^{n-1} \binom{s}{j} \\
&=& k+1 + (k+1)\downarrow_2 \binom{n}{2} + \sum_{j=3}^k (k+1)\downarrow_j
\binom{n}{j}.
\end{eqnarray*}
\end{proof}

We next consider simultaneous bi-avoidance of the patterns (1-2, 1~0)
and (1-2, 0~1).

\begin{theorem} Let $\Upsilon = \{(1\mbox{-}2, 1~0), (1\mbox{-}2, 0~1)\}$
and let $A_k^{\Upsilon}(t) = \sum_{n \geq
0}Av_{n,k}^{\Upsilon}t^n$ and $C(t) = \sum_{n \geq1}n!t^n$. Then
\begin{equation}\label{Uordinarygf}
A_k^{\Upsilon}(t) = \frac{1+C(t)}{1-(k-1)C(t)}.
\end{equation}
\begin{proof}
Fix $k$ and suppose that $(\sg,w)$ is an element of $C_k \wr S_n$
which bi-avoids both (1-2, 1~0) and (1-2, 0~1). Now if $w = i^n$ is
constant, then clearly $\sg$ can be arbitrary so that there are $n!$
such elements.

Next assume that $w$ is not constant so there is an $s \geq 2$ such
that $w = i_1^{a_1} i_2^{a_2} \cdots i_s^{a_s}$ where $a_r \geq 1$
for $r =1, \ldots, s$ and $i_r \neq i_{r+1}$ for $r =1, \ldots,
s-1$. We claim that $\sg_1 \cdots \sg_{a_1}$ must be the $a_1$
largest elements of $\{1,\ldots,n\}$. If not, then let $n-r$ be the
largest element which is not in $\sg_1 \cdots \sg_s$. Thus $r <
a_1-1$ and there must be at least one $\sg_t$ with $t \leq s$ such
that $\sg_t < n-r$. Now it cannot be that $\sg_{a_1+1} =n-r$ since
otherwise $(\sg_t~\sg_{a_1+1},i_1~i_2)$ would be an occurrence of
either (1-2,1~0) or (1-2,0~1). Hence it must be the case that
$\sg_{a_1+1} < n-r$ and $\sg_p =n-r$ for some $p > a_1+1$.  But then
no matter what color we choose for $w_p$, either
$(\sg_t~\sg_{p},i_1~w_p)$ or $(\sg_{a_1+1}~\sg_{p},i_2~w_p)$ would
be an occurrence of either (1-2,1~0) or (1-2,0~1). We can continue
this reasoning to show that for any $1 \leq p < q \leq s$, the
elements of $\sg$ corresponding to the block $i_p^{a_p}$ in $w$ must
be strictly larger than the elements of $\sg$ corresponding to the
block $i_q^{a_q}$  in $w$. This given, it follows that we can
arrange the elements of $\sg$ corresponding to a block $i_p^{a_p}$
in $w$ in any way that we want and we will always produce a pair
$(\sg,w)$ that bi-avoids  both (1-2,1~0) and (1-2,0~1). Thus for
such a $w$, we have $k(k-1)^{s-1}$ ways to choose the colors $i_1,
\ldots, i_s$ and $a_1! a_2! \cdots a_k!$ ways to choose the
permutation $\sg$. It follows that
\begin{equation}
Av_{n,k}^{\Upsilon} = k~n! + \sum_{s=2}^{n-1} \sum_{a_1 + \cdots + a_s =n, a_i > 0} k(k-1)^{s-1} a_1! a_2! \cdots a_s!
\end{equation}
which is equivalent to (\ref{Uordinarygf}).
\end{proof}
\end{theorem}

Finally, we end this section by considering the number of
elements of $C_k \wr S_n$ which bi-avoid both (1-2,0~0) and (1-2,0~1).
In this case, we only have a result for the case when $k=2$.

\begin{theorem}\label{cat} The number of permutations in $C_2 \wr S_n$
simultaneously bi-avoiding (1-2,0~0) and (1-2,0~1) is given by the
$(n+1)$-th Catalan number $C_{n+1}=\frac{1}{n+2}{2n+2 \choose n+1}$.
\end{theorem}

\begin{proof} We prove the statement by establishing a bijection between the objects in question of length $n$ and
the {\em Dyck paths} of semi-length $(n+1)$ known to be counted by
the Catalan numbers (A Dyck path of semi-length $n$ is a lattice
path from (0,0) to $(2n,0)$ with steps (1,1) and $(1,-1)$ that never
goes below $x$-axis).

Suppose $(\sg,w)\in C_2\wr S_n$. Note that $\sg$ must bi-avoid the
pattern 1-2-3, as in an occurrence of such pattern in $\sg$, there
are two letters of the same color leading to an occurrence of
(1-2,0~0). Thus, the structure of $\sg$, as it is well-known, is two
decreasing sequences shuffled.

Subdivide $\sg$ into so called {\em reverse irreducible components}.
A reverse irreducible component is a factor $F$ of $\sg$ of minimal
length such that everything to the left (resp. right) of $F$ is
greater (resp. smaller) than any element of $F$. For example, the
subdivision of $\sg=6574312$ is $\sg=657-4-3-12$. The blocks of size
1 are {\em singletons}. In the example above, 4 and 3 are
singletons. This is easy to see, that any singleton element in $\sg$
can have any color (either 0 or 1). We will now show that the color
of each element of a non-singleton block is uniquely determined.

Indeed, irreducibility of a single block means that two decreasing
sequences in the structure of (1-2-3)-avoiding permutations are the
block's sequence of left-to-right minima and the block's sequence of
right-to-left maxima which do {\em not} overlap. Thus, for any
left-to-right minimum element $x$ (except possibly the last
element), one has an element $y$ greater than $x$ to the right of it
inside the same block and vice versa, from which we conclude that
$x$ must receive color 1, whereas $y$ must receive color 0
(otherwise a prohibited pattern will bi-occur).

We are ready to describe our bijection. For a given $(\sg,w)\in
C_2\wr S_n$, consider the matrix representation of $\sg$, that is an
integer grid with the opposite corners in (0,0) and $(n,n)$, and
with a dot in position $(i-\frac{1}{2},\sg_i-\frac{1}{2})$ for
$i=1,2,\ldots,n$. We will give a description of a path $P$
(corresponding to $(\sg,w)$) from $(0,n+1)$ to $(n+1,0)$ involving
only steps $(0,-1)$ and (1,0) that never goes above the line
$y=-x+n+1$. Clearly, $P$ can be transformed to a Dyck path of length
$(n+1)$ by taking a mirror image with respect to the line
$y=-x+n+1$, rotating 45 degrees counterclockwise, and making a
parallel shift.

To build $P$, set $i:=0$ and $j:=n+1$, and do the following steps
letting $P$ begin at $(i,j)$. Clearly, each reverse irreducible
block of $\sg$ defines a square on the grid which is the matrix
representation of the block. We call the reverse irreducible block
of $\sg$ with the left-top corner at $(i,j)$ the {\em current
block}.

\begin{itemize}
\item[Step 1.] If the current block is {\em not} a
singleton, go to Step 2. If the color of the element with
$x$-coordinate equal $i+\frac{1}{2}$ is $0$ (resp. 1) travel around
the current block counterclockwise (resp. clockwise) to get to the
point $(i+1,j-1)$. Note that $P$ touches the line $y=-x+n+1$ if the
color is 1. Set $i:=i+1$ and $j:=j-1$, and proceed with Step 3.

\item[Step 2.]
In Step 2 we follow a standard bijection between (1-2-3)-avoiding
permutations and Dyck paths that can be described as follows. Let
$(k,\ell)$ be the point of the current block opposite to $(i,j)$.
Start going down from $(i,j)$ until the $y$-coordinate of the
current node gets $\frac{1}{2}$ less than the $y$-coordinate of the
dot with $x$-coordinate equal $i+\frac{1}{2}$. Start moving
horizontally to the right and go as long as possible making sure
that none of the dots are below the part of $P$ constructed so far
and $i\leq k$ and $j\leq \ell$. Suppose $(i_1,j_1)$ is the last
point the procedure above can be done (that is, we were traveling on
the line $y=j_1$ and either $i_1=k$ and $j_1=\ell$ or there is a dot
with $x$-coordinate $i_1+\frac{1}{2}$ having $y$-coordinate less
than $j_1$). If $i:=k$ and $j:=\ell$, proceed with Step 3; otherwise
set $i:=i_1$ and $j:=j_1$ and go to Step 2. Note that in Step 2, $P$
never touches the line $y=-x+n+1$.

\item[Step 3.] If $j=0$, make as many as it takes horizontal steps to get to the point $(n+1,0)$ and terminate; otherwise go to Step 1.
\end{itemize}

Returning to our example, $\sg = (6~5~7~4~3~1~2, 1~1~0~1~0~1~0)$, we
have given the matrix diagram and outlined the reverse irreducible
blocks in Figure~\ref{bij2}.
We start our path at $(0,8)$. We travel down until we reach $(0,6)$,
when we are $\frac{1}{2}$ less than the y-coordinate of our first
point $(\frac{1}{2}, 6\frac{1}{2})$. We then continue traveling
right and down as described in Step 2. We travel clockwise around
our singleton colored $1$, and counterclockwise around our singleton
colored $0$. Then we continue to the final reverse irreducible block
and finish our path, given in Figure~\ref{bij3}. The resulting path
is presented in Figure~\ref{bij4}.

%

\begin{figure}[htp]
\centering
\includegraphics[width=100pt]{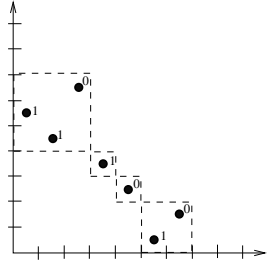}\\
\caption{Matrix representation of $\sg = (6~5~7~4~3~1~2,
1~1~0~1~0~1~0)$ with reverse irreducible blocks
outlined.}\label{bij2}
\end{figure}

\begin{figure}[htp]
\centering
\includegraphics[width=100pt]{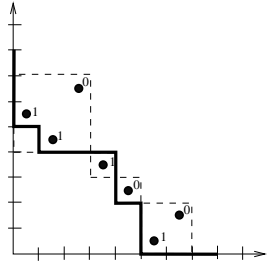}\\
\caption{The Dyck path corresponding to $\sg = (6~5~7~4~3~1~2,
1~1~0~1~0~1~0)$ added to Figure \ref{bij2}.}\label{bij3}
\end{figure}

\begin{figure}[htp]
\centering
\includegraphics[width=100pt]{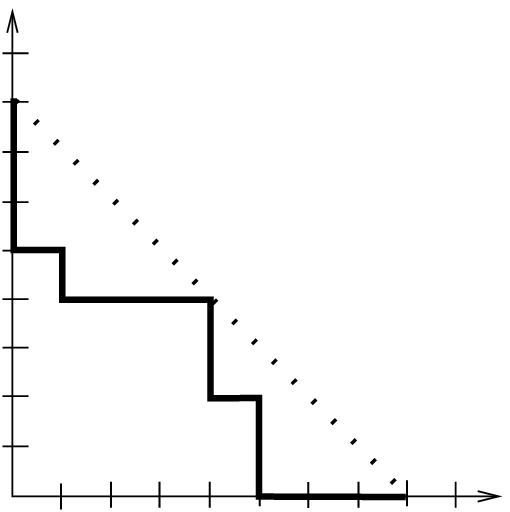}\\
\caption{The Dyck path corresponding to $\sg = (6~5~7~4~3~1~2,
1~1~0~1~0~1~0)$. Note the path only touches the line $y = -x + n +
1$ after a singleton colored $1$.}\label{bij4}
\end{figure}
\end{proof}

\section{Concluding remarks}

A natural way to extend work in this paper is to consider
bi-occurrences of patterns in $C_k \wr S_n$ of length more than 2.
However, there are questions left regarding the pattern (1-2,0~1),
which are listed below starting from the most ambitious one:

\begin{itemize}
\item Find distribution of (1-2,0~1) on $C_k \wr S_n$.
\item For $k\geq 3$, find the number of permutations in $C_k \wr
S_n$ bi-avoiding (1-2,0~1). The number of permutations in $C_k \wr
S_n$ bi-avoiding (1-2,0~1) for initial values of $k$ and $n$ are as
follows:
$$\begin{array}{ll}
k=3: & 3, 15, 101, 842, 8302, \ldots \\
k=4: & 4, 26, 224, 2361, \ldots \\
k=5: & 5, 40, 420, 5355, \ldots.
\end{array}$$
\item OEIS~\cite[A002720]{Sloane} suggests that the number of permutations in $C_2 \wr S_n$ bi-avoiding (1-2,0~1)
(for initial values of $n$ these numbers are 2, 7, 34, 209, 1546,
13327,$\ldots$) is the same as the number of
\begin{itemize}
\item partial permutations of an $n$-set;
\item $n \times n$ binary matrices with at most one 1 in each row
and column; \item matchings in the bipartite graph $K(n,n)$.
\end{itemize}
It would be interesting to find combinatorial proofs for the
conjectures above.
\end{itemize}

\end{document}